\documentclass[10pt]{article}

\font\smallit=cmti10
\font\smalltt=cmtt10

\usepackage{amssymb,latexsym,amsmath,epsfig,amsthm, url}

\makeatletter
\renewcommand\section{\@startsection {section}{1}{\z@}%
                                   {-3.5ex \@plus -1ex \@minus -.2ex}%
                                   {2.3ex \@plus.2ex}%
                                   {\normalfont\normalsize\bfseries\boldmath}}
\renewcommand\subsection{\@startsection{subsection}{2}{\z@}%
                                     {-3.25ex\@plus -1ex \@minus -.2ex}%
                                     {1.5ex \@plus .2ex}%
                                     {\normalfont\normalsize\bfseries\boldmath}}
\renewcommand{\@seccntformat}[1]{\csname the#1\endcsname. }
\makeatother

\theoremstyle{plain}
\newtheorem{theorem}{Theorem}
\newtheorem{lemma}{Lemma}

\newtheorem{corollary}{Corollary}

\theoremstyle{definition}
\newtheorem{definition}{Definition}

\theoremstyle{remark}
\newtheorem{remark}{Remark}

\newcommand{\refk}[1]{(\ref{#1})} 

\begin{document}

\begin{center}
\uppercase{\bf \boldmath Ternary Digits of Powers of Two}

{\bf Xuyi Ren\footnote{Undergraduate author}} \\
{\smallit Department of Mathematics, Grinnell College, Grinnell, Iowa, USA} \\
{\tt renxuyi@grinnell.edu} \\ 
\vskip 20pt
{\bf Christian Roettger} \\
{\smallit Department of Mathematics, Iowa State University, Iowa, USA} \\
{\tt roettger@iastate.edu} \\ 
\vskip 10pt
\end{center}


\centerline{\bf Abstract}

The \textit{ternary digits of $2^n$} are a finite sequence of 0s, 1s, and 2s. 
It is a natural question to ask whether the frequency of any string of 0s, 1s, and 2s 
in this sequence approaches the same limit for all strings of the same length, as the exponent $n$ approaches infinity (\textit{Uniform Distribution in the limit}).

Currently the answer to this question is unknown. Even a much weaker conjecture by Erd\"os is still open. But we present computational results (up to $n = 10^6$) supporting uniform distribution in the limit. 

In this context, we discuss implications of Benford's Law and a special case of Baker's Theorem. 

Then we investigate the infinite sequence of ternary digits of $\log_3(2)$. There are analogous questions about the distribution of strings of 0s, 1s, and 2s in that sequence. If there is uniform distribution in the limit, then $\log_3(2)$ is called \textit{normal to base 3}. 

In the absence of definitive results, we can offer again computational evidence from the first $10^6$ ternary digits of $\log_3(2)$, strongly supporting the conjecture that $\log_3(2)$ is normal to base 3. 

\noindent

\pagestyle{myheadings}
\markright{\smalltt Ternary Digits 11/2025)\hfill} 
\thispagestyle{empty}
\baselineskip=12.875pt
\vskip 30pt


\section{Frequency of Ternary Digits of Powers of Two} 

Representing powers of 2 in base 3 means writing 
\[  2^n = \sum_{i=0}^{k} c_i 3^i
\]
with a finite sequence of \textit{ternary digits} $c_i=0,1,2$ and the \textit{leading digit} $c_k\neq0$. 

Surprisingly little is known about the asymptotic behavior of the frequency of 0s, or 1s, or 2s in this sequence as $n$ tends to infinity. 

\newcommand{\dcount}{c_d(n)} 
\newcommand{\Dcount}{C_d(N)} 
\newcommand{\len}{l(n)} 
\newcommand{\Len}{L(n)} 
\newcommand{\freq}{f_d(n)} 
\newcommand{\Freq}{F_d(N)} 
\newcommand{\dcounthigh}{c_d(n,H)} 
\newcommand{\Dcounthigh}{C_d(N,H)} 
\newcommand{\Freqhigh}{F_d(N,H)} 
\newcommand{\scount}{c_s(n)} 
\newcommand{\Scount}{C_s(N)} 
\newcommand{\SFreq}{F_s(N)} 
\newcommand{\hcount}{c_d(n,H)} 
\newcommand{\limHFreq}{L_{d,H}} 
\newcommand{\phid}[2]{\phi_{#1}(#2)}

Let us write $\lceil x\rceil$ for the smallest integer greater or equal to $x$, and 
$\alpha=\log_3 2$. Then $\len=\lceil n\alpha\rceil$ is the number of ternary digits of $2^n$.
For $d=0,1,2$ define $\dcount$ to be the \textit{count of ternary digits equal to $d$ in $2^n$}, and the \textit{frequency of $d$} by
\begin{eqnarray*}
    \freq &=& \frac{\dcount}{\len}\\
\end{eqnarray*}

\medskip
\noindent\textit{General counting function.}
For any integer $A\ge1$ and $d\in\{0,1,2\}$, let $\phi_d(A)$ denote the number of ternary digits of $A$ that are equal to $d$.
For powers of two we keep the shorthand $c_d(n):=\phi_d(2^n)$.
When a statement applies to arbitrary integers (e.g., Theorem~2), we will write $\phi_d(A)$; for powers of two we use $c_d(n)$.
\medskip

We are now ready to state several conjectures, from strongest to weakest, about how close the distribution of frequencies $\freq$ comes to being uniform, as $n$ grows to infinity. 

\begin{enumerate}
\item[C1]
For $d=0, 1, 2$, the frequency $\freq$ of ternary digits equal to $d$ has limit $1/3$ as $n$ goes to infinity (\textit{uniform distribution in the limit}). 
\item[C2]
For $d=0, 1, 2$, the frequency $\freq$ has a nonzero limit as $n$ goes to infinity. 
\item[C3]
For $d=0,1,2$, the frequency $\freq$ has a nonzero lower bound valid for large $n$.
\item[C4]
 (Erd\"os) Every power $2^n$ with $n > 8$ has at least one ternary digit equal to 2. 
\end{enumerate}

These conjectures are meant to illustrate the gulf between what seems plausibly true and what is known. 
In the words of Terry Tao, even conjecture C4 is 'still a fair distance beyond what one can do with current technology'
\cite{Tao2011_2}. See Lagarias \cite{Lagarias} for some results concerning this conjecture. 

In Section 2, we start by considering the distribution of \textit{aggregate frequencies} $\Freq$, defined using aggregate count $\Dcount$ and total number of digits $\Len$, 
\begin{eqnarray*}
    \Dcount &=&\sum_{n=1}^N \dcount\\
    \Len &=& \sum_{n=1}^N \len\\
    \Freq &=& \frac{\Dcount}{\Len}
\end{eqnarray*}
We can show that conjecture C1 would imply for all $d=0,1,2$
\begin{equation}\label{equation:AGGFREQ}
    \lim_{N\to\infty} \Freq = \frac{1}{3}
\end{equation}
So Equation \refk{equation:AGGFREQ} can also be considered to be a weaker conjecture than C1.

\begin{lemma}
If conjecture C1 is true, then Equation~\refk{equation:AGGFREQ}.
holds.
\end{lemma}

\begin{proof}
By definition, $F_d(N)$ is a weighted average of the frequencies $f_d(n)$ with nonnegative weights $\ell(n)/\Len$. If each $f_d(n)$ converges to $1/3$ as $n\to\infty$ (Conjecture C1), then the weighted average also converges to $1/3$.
\end{proof}

We present computational evidence for Equation \refk{equation:AGGFREQ}. Then we study a refinement using blocks of digits. 
Suppose the string of ternary digits of $2^n$ is cut up into blocks of length $k$ (possibly with a string of fewer than $k$ digits remaining at the end). Let $B_k(n):=\lfloor \len/k\rfloor$ be the number of such blocks, and for a string $s$ of 0s, 1s, and 2s, let $\scount$ be its non-overlapping count, with the aggregate version $\Scount=\sum_{n=1}^N \scount$ and aggregate frequency 
\begin{equation}\label{SFREQDEF}  
    \SFreq = \frac{\Scount}{\sum_{n=1}^N B_k(n)}
\end{equation}
The original conjecture C1 was motivated by the apparent randomness of the digits of $2^n$. If they really behaved as if they were drawn at random, then any string of length $k$ would occur with probability $1/3^k$. 
So it is natural to conjecture that this should be the limit of the aggregate frequencies $\SFreq$. 
After presenting our results about frequencies of strings of length 2 and 3, we end
Section 2 with data about the strongest conjecture C1. 

We then ask what, if anything, we can actually prove about the distribution of digits. 
Well-known results like Benford's Law and Baker's Theorem have implications for these conjectures, but they neither prove nor disprove them. 
We show in Sections 3 and 4, respectively, how to adapt these theorems to our situation, then examine the interplay with the digit frequencies. 

In Section 5, we explore the relationship of these conjectures to the ternary digits of the number
\[  \alpha = \log_3(2) \approx 0.63093\dots
\]
This number plays already a key role in Sections 3 and 4. The concept of a \textit{normal number (to base 3)} is again about the distribution of digits 0, 1, 2 in the ternary representation of that number, in our case
\[  \alpha = \sum_{j=0}^\infty d_j 3^{-j}
\]
A number is called \textit{normal to base 3} if the frequency of any fixed string of length $k$ among the first $r$ length-$k$-blocks of ternary digits of that number approaches $1/3^k$, as $r$ approaches infinity. 
Currently, it is unknown whether $\alpha$ is normal to base 3. 

Despite the obvious connections between the sequence of ternary digits of $\alpha$ and ternary digits of powers of 2, conjectures about the one do not seem to imply conjectures about the other. We can give a heuristic explanation for this non-connection, although it is impossible to prove the absence of any such implication. 

At the end of Section 5, we present computational evidence suggesting that $\alpha$ is indeed normal to base 3. 

In the concluding Section 6, we discuss the relationship of ternary digits of powers of 2 to another famous conjecture -- Selfridge's conjecture about integer complexity.

\section{Computational Evidence for Uniform Distribution}

To investigate the conjectures outlined in Section 1, we performed a computational analysis for powers of two with exponent $n$ in the range $1 \le n \le 10^6$. We gathered data on the distribution of ternary digits and strings of digits, observing whether their frequencies approach uniform distribution as $n$ becomes large. The entire computation required 2 hours, 51 minutes, and 47 seconds of processing time.

\subsection{Methodology}

The calculations were carried out using a custom program written in C, leveraging the GNU Multiple Precision Arithmetic Library (GMP) to handle the integers that would cause an overflow. For each integer $n$ from 1 to $10^6$, the program performed the following steps:

\begin{enumerate}
     \item Compute the value of $2^n$ using GMP's arbitrary-precision integer functions.
     \item Convert the resulting integer into its base-3 string representation $S_n$.
     \item Tally the occurrences of the individual digits '0', '1', and '2' within $S_n$.
     \item For string lengths $k=2$ and $k=3$, parse $S_n$ into non-overlapping blocks of length $k$. Tally the occurrences of each of the $3^k$ possible strings (e.g., for $k=2$, count '00', '01', '02', \dots, '22').
\end{enumerate}

The counts for both individual digits and digit strings were aggregated across all $n$. The total number of digits processed in this computation was 315,465,692,249.

\subsection{Results for Aggregate Digit Frequencies}
\label{sub:aggregate}

We conjectured that the aggregate frequency $\Freq$ of each digit converges towards $1/3$. 
Figure \ref{fig:agg} shows the deviation of these aggregate frequencies from the conjectured limit $1/3$. 

\begin{figure}[h!]
  \centering
  \includegraphics[width=\linewidth]{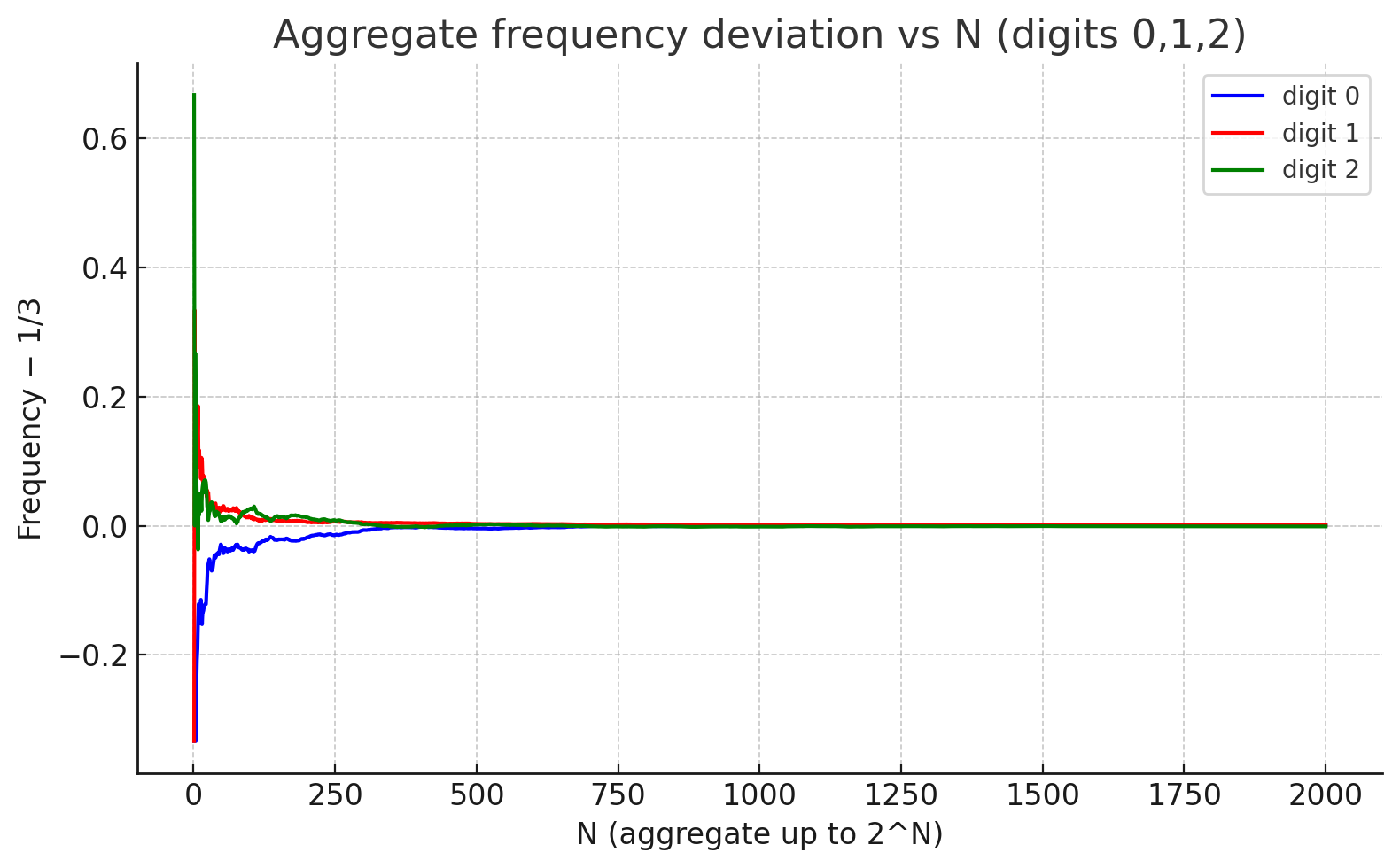} 
  \caption{Deviation of aggregate digit frequencies $\Freq$ from $1/3$
           for exponents $N\le 2000$. Colors: blue $=0$, orange $=1$, green $=2$.}
  \label{fig:agg}
\end{figure}

We chose not to plot the deviation for larger values of $N$, because the plot would look very similar -- three flat lines which are visually indistinguishable from each other for $N>750$. 

But we did compute aggregate frequencies up to $N = 10^6$. The final results are summarized in 
Table \ref{tab:digit_freq_agg}. 

\begin{table}[h!]
\centering
\begin{tabular}{|c|c|}
\hline
\textbf{Digit $d$} & \textbf{Aggregate Frequency $\Freq$ in percent} \\ \hline
\textbf{0} & 33.333041\% \\
\textbf{1} & 33.333576\% \\
\textbf{2} & 33.333382\% \\ \hline
\end{tabular}
\caption{Frequency of Ternary Digits in powers of $2$, aggregated up to exponent $N=10^6$.}
\label{tab:digit_freq_agg}
\end{table}

The percentages come ever closer to the theoretical value of $33.\overline{3}\%$. 

Both Figure \ref{fig:agg} and Table \ref{tab:digit_freq_agg} strongly support the conjecture that each of 0, 1, 2 appear with equal frequency in the limit.

\subsection{Aggregate Digit String Frequencies}

A more refined test of uniform distribution is to examine the frequency of strings of digits as defined in Equation \refk{SFREQDEF}. 
We conjectured that the frequency of any string of length $k$, aggregated up to exponent $N$, would approach $3^{-k}$ as $N$ grows to infinity. 
Our analysis for strings of length $k=2$ and $k=3$ supports this conjecture. 
Table \ref{tab:string_freq_len2} shows the aggregate frequencies for strings of length 2 and $N = 10^6$.

\begin{table}[h!]
\centering
\begin{tabular}{|c|c||c|c|}
\hline
\textbf{String} & \textbf{Frequency} & \textbf{String} & \textbf{Frequency} \\ \hline
'00' & 11.110880\% & '12' & 11.111239\% \\
'01' & 11.111071\% & '20' & 11.111079\% \\
'02' & 11.111008\% & '21' & 11.111290\% \\
'10' & 11.111271\% & '22' & 11.111047\% \\
'11' & 11.111114\% & & \\
\hline
\end{tabular}
\caption{Aggregate Frequency of Strings of Length 2. The expected frequency is $1/9 \approx 11.111\%$.}
\label{tab:string_freq_len2}
\end{table}

The results for strings of length 3 and $N = 10^6$, shown in Table \ref{tab:string_freq_len3}, were similarly close to the expected frequency of $1/27 \approx 3.703704\%$.

\begin{table}[h!]
\centering
\begin{tabular}{|c|c||c|c||c|c|}
\hline
\textbf{String} & \textbf{Frequency} & \textbf{String} & \textbf{Frequency} & \textbf{String} & \textbf{Frequency} \\ \hline
'000' & 3.703532\% & '100' & 3.703663\% & '200' & 3.703700\% \\
'001' & 3.703761\% & '101' & 3.703772\% & '201' & 3.703796\% \\
'002' & 3.703652\% & '102' & 3.703779\% & '202' & 3.703696\% \\
'010' & 3.703561\% & '110' & 3.703813\% & '210' & 3.703712\% \\
'011' & 3.703825\% & '111' & 3.703629\% & '211' & 3.703716\% \\
'012' & 3.703665\% & '112' & 3.703635\% & '212' & 3.703820\% \\
'020' & 3.703620\% & '120' & 3.703779\% & '220' & 3.703632\% \\
'021' & 3.703645\% & '121' & 3.703750\% & '221' & 3.703807\% \\
'022' & 3.703714\% & '122' & 3.703727\% & '222' & 3.703600\% \\
\hline
\end{tabular}
\caption{Aggregate Frequency of Ternary Strings of Length 3. The expected frequency is $1/27 \approx 3.7037\%$.}
\label{tab:string_freq_len3}
\end{table}

The rapid convergence of the frequencies for both individual digits and short strings of digits to their theoretical uniform values provides substantial computational evidence in support of our conjectures. 

\subsection{Variance and Standard Deviation}
\label{sub:variance}

The computational evidence above suggests that the ternary digits of powers of 2 'behave like' a random sequence. But so far, we have only looked at averages, which would correspond to the mean of the theoretical distribution. It is natural to ask about the standard deviation -- \textit{how close} should we expect these aggregate averages to be to the mean? Suppose that the digits in question really were independently drawn from the uniform distribution at random. Since we are looking at behavior for large exponents, we can neglect the fact that the first and last digit can never be zero. Then the frequencies of 0, 1, and 2 in the $\len$ digits of any individual power $2^n$ would all follow a Binomial Distribution $Bi(\len,p)$ with parameters $p=1/3$ and $\len = \lceil\log_3(2^n)\rceil$ independent trials.
Using $\log_3(2^n) = n\alpha$, the expected value of the aggregate count is 
\[  E[X] = \sum_{n=1}^N \len p
\]
Dividing by the total $\Len=\sum_{n=1}^N \len$ gives the theoretical expected value of the aggregate frequency $\bar{X}$ as $1/3$. For the variance, we know that $Bi(\len,p)$ has variance $\len p(1-p)$, aggregating this gives the variance of the aggregate count
\[  V[X] = \Len p(1-p)
\]
Dividing $X$ by the total, we get the theoretical variance of the aggregate frequency
\[  V[\bar{X}] = V\left[  \frac{X}{\Len}\right] = \frac{p(1-p)}{\Len}
\]
With $p = 1/3$ and $\lceil x\rceil\approx x$, we can approximate 
\begin{equation}\label{STDEV}
    V[\bar{X}] \approx \frac{4}{9\alpha N(N+1)}
\end{equation}
Take the square root to get the (approximate) standard deviation $\sigma$. 
Eg with $N = 10^6$, we get $\sigma\approx 8.4 \cdot 10^{-7}$. 

Take the square root to get the (approximate) standard deviation $\sigma$.
For $N=10^6$, using the \emph{exact} totals
\[
\sum_{n\le N}\ell(n)=315{,}465{,}692{,}249
\]
\[
M_2=\sum_{n\le N}\Big\lfloor\frac{\ell(n)}{2}\Big\rfloor=157{,}732{,}596{,}126,\\
\]
\[
M_3=\sum_{n\le N}\Big\lfloor\frac{\ell(n)}{3}\Big\rfloor=105{,}154{,}897{,}417,
\]
we obtain
\[
\sigma_{\text{digit}}
=\sqrt{\frac{p(1-p)}{\sum\ell(n)}}
\approx 8.393\times 10^{-7},\qquad
\]
\[
\sigma_{k=2}
=\sqrt{\frac{(1/9)(8/9)}{M_2}}
\approx 7.913\times 10^{-7},\qquad
\]
\[
\sigma_{k=3}
=\sqrt{\frac{(1/27)(26/27)}{M_3}}
\approx 5.824\times 10^{-7}.
\]
These benchmarks are only slightly smaller than the empirical deviations in
Tables~\ref{tab:digit_freq_agg}, \ref{tab:string_freq_len2}, and \ref{tab:string_freq_len3},
indicating that the aggregate data are consistent with simple i.i.d.\ noise.
Since there are $m=3,9,27$ categories, the largest deviation across categories is
naturally a few $\sigma$ (heuristically on the order of $\sigma\sqrt{2\ln m}$), so
it is not surprising to observe the maximum curve a little farther from the mean than the average one.


\subsection{Non-aggregate digit tallies}\label{ssec:nonagg}
Since the above results of counts aggregated over all exponents
$1\le n\le N$ are so close to uniform distribution, we were tempted to examine the original conjecture C1 -- 
the conjecture that the frequencies of 0, 1, 2 within the ternary digits of \textit{individual powers $2^n$} all have the limit $1/3$. 
For every $n\le2000$ we computed
\[  \freq = \frac{\dcount}{\len}
\]
and plotted the deviation $\freq -1/3$.
Figure \ref{fig:nonagg} visualizes the result.

\begin{figure}[h!]
  \centering
  \includegraphics[width=\linewidth]{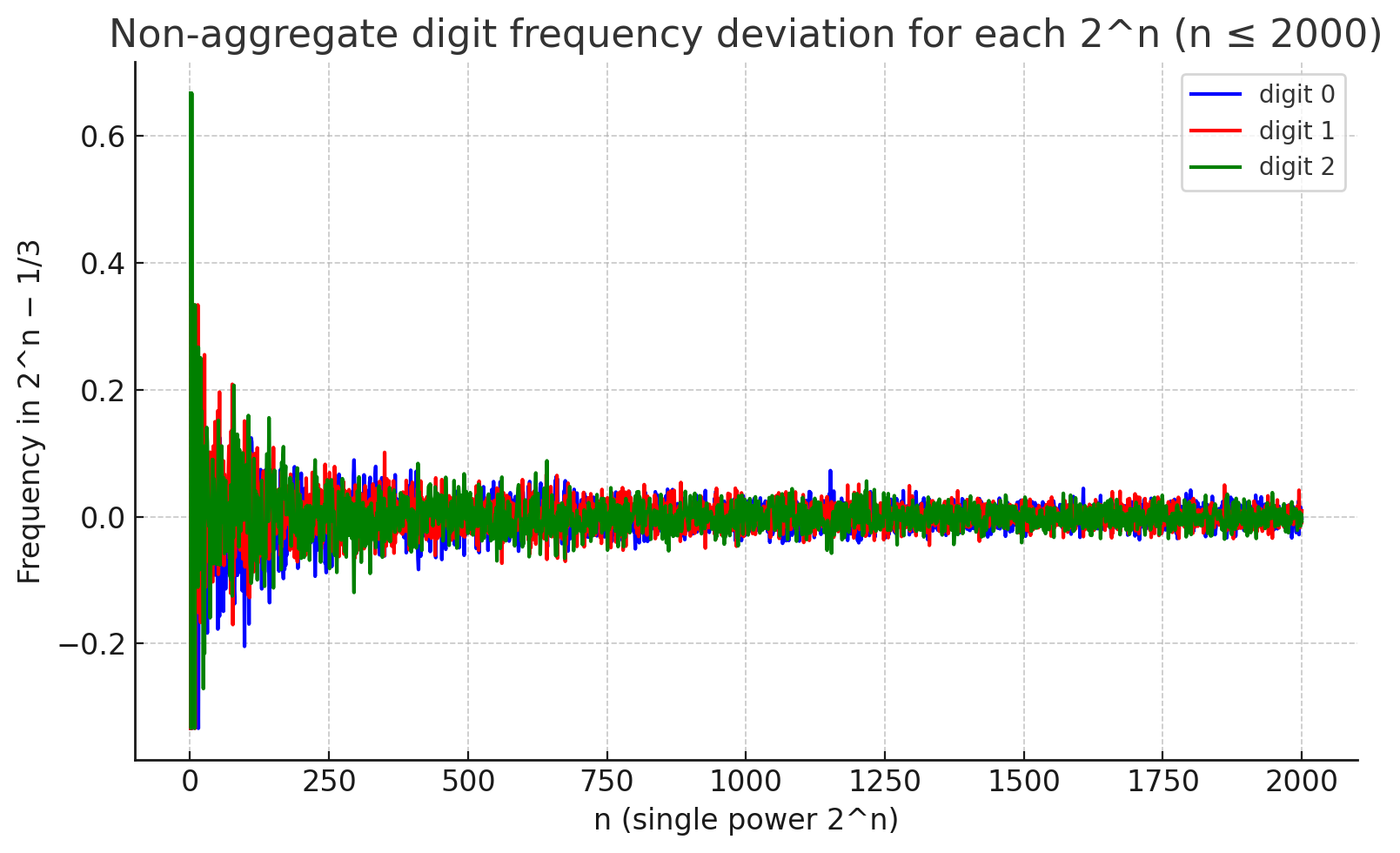}
  \caption{Deviation of the frequency $\freq$ of each digit in
           \emph{individual} powers $2^n$ from the uniform value $1/3$
           for $1\le n\le2000$.  Colors are consistent with
           Figure~\ref{fig:agg}: blue $=0$, orange $=1$, green $=2$.}
  \label{fig:nonagg}
\end{figure}

Two qualitative features stand out: 

\begin{enumerate}
  \item \textbf{Damped oscillations.}
  Early-$n$ fluctuations are on the order of $10^{-2}$ and decrease; by $n\approx 500$
  the deviations typically lie below $5\times 10^{-3}$.
  \item \textbf{No clear digit ordering.}
  All three digits exhibit similar fluctuation patterns. A useful benchmark is the
  i.i.d.\ band $\pm 2\sqrt{p(1-p)/\ell(n)}$ with $p=1/3$, within which most points fall.
\end{enumerate}

To complement the small-$n$ view of Figure~\ref{fig:nonagg}, 
Table~\ref{tab:nonagg_n_1e6} records the \emph{non-aggregate} digit frequencies for the single large exponent $n=10^6$.
The ternary expansion of $2^{10^6}$ has length $L=\lceil 10^6\log_3 2\rceil = 630{,}930$ digits, and each frequency is very close to $1/3$.

\begin{table}[h!]
\centering
\begin{tabular}{|c|r|r|}
\hline
\textbf{Digit} & \textbf{Count} & \textbf{Percentage} \\ \hline
0 & 210{,}367 & 33.342368\% \\
1 & 209{,}942 & 33.275007\% \\
2 & 210{,}621 & 33.382626\% \\ \hline
\textbf{Total} & 630{,}930 & 100.000000\% \\ \hline
\end{tabular}
\caption{Non-aggregate digit counts for the single exponent $n=10^6$.}
\label{tab:nonagg_n_1e6}
\end{table}

Compared to the aggregate results of Section~\ref{sub:aggregate},
Figure~\ref{fig:nonagg} and Table~\ref{tab:nonagg_n_1e6} offer some support of Conjecture~C1: the digit
distribution appears to converge to uniformity already within individual samples.
The random oscillations decay, but much more slowly than in the aggregate situation.
In general, this behavior can be expected — the evidence in favor of uniform distribution is just weaker.  
A computation similar to Section~\ref{sub:variance} gives the standard deviation of the non-aggregate frequencies $\bar{x}$, if the digits were randomly drawn from a uniform distribution, as  
\[
  \sigma(\bar{x}) = \sigma\left(\frac{x}{\lceil n\alpha\rceil}\right) 
    \approx \sqrt{\frac{p(1-p)}{n\alpha}},
\]
which, for $n = 2000$ and $n=10^6$, gives $\sigma\approx 0.013$ and $\sigma\approx 0.0006$, respectively.  
Many of the empirical deviations we see in Figure~\ref{fig:nonagg} and Table \ref{tab:nonagg_n_1e6} are even smaller than this.

A computation similar to Section~\ref{sub:variance} gives the i.i.d.\ benchmark
for a single power $2^n$:
\[
  \sigma\!\left(\frac{c_d(n)}{\ell(n)}\right)
  = \sqrt{\frac{p(1-p)}{\ell(n)}}
  \quad\text{with}\quad p=\tfrac13,\ \ \ell(n)=\lceil n\alpha\rceil.
\]
Numerically,
for $n=2000$ (length $\ell=1262$) we get $\sigma\approx 1.32698\times 10^{-2}$,
and for $n=10^6$ (length $\ell=630{,}930$) we get $\sigma\approx 5.93476\times 10^{-4}$.
The observed deviations in Figure~\ref{fig:nonagg} and Table~\ref{tab:nonagg_n_1e6}
are again comparable to those computed for the simple i.i.d.\ model.

\section{Uniform Distribution and Benford's Law}
\label{section:BENFORD}

A good starting point for actually proving results about digit distribution is to consider the frequency of the leading digit. A well-known empirical observation, first made by Simon Newcomb in 1881 and later popularized by Frank Benford, is that the leading digits in many real-world datasets are not uniformly distributed \cite{Hill1998}. Instead, they tend to follow a logarithmic distribution known as Benford's Law, which gives the probability of a highest decimal digit 
$d=1, 2, \dots, 9$ as
\[ P(d) = \log_{10}\left(d+1\right)  - \log_{10}\left(d\right) 
\]
This law predicts that '1' appears as the leading digit about 30.1\% of the time, while '9' appears less than 5\% of the time.

The theoretical underpinning of Benford's Law is the theory of uniform distribution modulo 1 (see eg \cite{Hlawka} or \cite{Niederreiter}). A sequence of positive numbers $(a_n)$ satisfies Benford's Law if the sequence of their base-10 logarithms, $(\log_{10}(a_n))$, is uniformly distributed modulo 1 \cite{Hill1998}. For the sequence of powers of two, $(2^n)$, we consider the logarithms $\log_{10}(2^n) = n \log_{10}(2)$. Since $\log_{10}(2)$ is an irrational number, the sequence of these logarithms is uniformly distributed modulo 1 by Weyl's Criterion. This proves that the sequence $(2^n)$ obeys Benford's Law in base 10.

We can adapt this reasoning to the base-3 context of our main problem. The leading ternary digit of $2^n$ is determined by the fractional part of $n \log_3(2)$. Specifically, the leading digit is '1' if the fractional part of $n\log_3(2)$ is in  $[0, \log_3(2))$, and '2' if it is in $[\log_3(2), 1)$. 
Since $\alpha = \log_3(2)$ is irrational, 
the sequence $(n\alpha)$ is uniformly distributed modulo 1. This implies a non-uniform distribution for the leading ternary digits '1' and '2'. 
The probabilities are thus $\log_3(2) \approx 63.1\%$ for '1' and $1 - \log_3(2) \approx 36.9\%$ for '2'.

Benford's Law can be adapted to strings of ternary digits as follows.

\begin{theorem}[Benford's Law for ternary digits]
\label{BENFORDTHM}
For any integer $m>0$, the frequency of powers of 2 with a leading string of ternary digits representing $m$ approaches $\log_3(m+1)-\log_3(m)$.
\end{theorem}

\begin{proof}
The leading digits of any number $A$ form a string which is the ternary representation of $m$ if and only if 
\begin{equation}\label{BENFORD1} A = 3^k m + r
\end{equation}
with $0\le r< 3^k$. 
So we can restate Theorem \ref{BENFORDTHM} as 
\[  \lim_{N\to \infty} \frac{1}{N} \#\{n\le N:\ 2^n = 3^km + r,\ 0 \le r<3^k\} = \log_3(m+1)-\log_3(m)
\]
Equation \refk{BENFORD1}, with the conditions on $r$, can then be further rewritten by taking logarithms with base 3,
\begin{equation}\label{BENFORD2} 
    k + \log_3(m) = \log_3(3^k m) \le \log_3(A) < \log_3(3^k(m+1)) = k + \log_3(m+1)
\end{equation}
We see that the leading digits of $A$ agree with $m$ if and only if $\log_3(A)$ falls into an interval of length $\log_3(m+1)-\log_3(m)$. Note that this length is always less than 1. The uniform distribution of $\log_3(2^n)$ modulo 1 then concludes the proof (see eg \cite{Hlawka} or \cite{Niederreiter}, but any textbook on uniform distribution will do -- they usually cover the case of $n\alpha$ modulo 1 as the very first example). 
\end{proof}

Here is an interesting consequence of Benford's Law: the average count of $d$ in the leading string of digits of $2^n$, $n=1,\dots, N$ approaches a limit as $N$ goes to infinity. 

\newcommand{\hcd}[1]{\gamma_d(#1)}
\newcommand{\HFreq}{F_{d,H}(N)}

To state the following theorem, we need notation for the \textit{count of $d=0,1,2$ in the leading digits} of arbitrary integers $A$, not just powers of 2. Let us write 
\begin{equation} \label{GAMMADEF}
    \hcd{A,H} = \text{number of $d$'s in the highest $(H+1)$ ternary digits of $A$}
\end{equation}
If $A$ has fewer than $H+1$ digits, let $\hcd{A,H}$ be the count of all digits equal to $d$. We define the \textit{average count in the highest $(H+1)$ digits of $A=2^n$} as 
\begin{equation}\label{HFREQDEF}
    \HFreq = \frac{1}{N} \sum_{n=1}^N \hcd{2^n,H}
\end{equation}

\begin{theorem}[Average count in leading digits]
\label{AVERAGEFREQTHM}
The average count of $d$ in the $(H+1)$ leading digits of $2^n$ approaches a limit as $N$ grows to infinity, 
\[  
\lim_{N\to\infty} \HFreq
= \sum_{3^H\le m<3^{H+1}} \hcd{m,H} [\log_3(m+1) - \log_3(m)] =: \limHFreq
\]
\end{theorem}

\begin{proof}
This follows directly from observing that strings of $H+1$ ternary digits, with nonzero leading digit, correspond exactly to numbers $m$ between $3^H$ and $3^{H+1}$ (the latter is excluded), and then applying Theorem \ref{BENFORDTHM}. 
Note that for small $n$, specifically those $n$ with $2^n < 3^{H}$, we do not have enough digits in $2^n$ to possibly match $m$. But this part of the aggregate count is bounded independently of $N$ and can therefore be neglected (the reader may have noticed the same issue already in Theorem \ref{BENFORDTHM}). 
\end{proof}

Dividing the average count by $H+1$ gives the \textit{average frequency of $d$ in the entire set of leading digits of $2^n$, $n=1,\dots, N$} (again neglecting small $n<H\alpha$). 
Theorem \ref{UNIFDISTRIBTHM} is what we would expect -- uniform distribution in the average frequencies in the leading string of digits, as the length $H+1$ of that string goes to infinity. 

\begin{theorem}[Uniform distribution of frequency in leading digits]
\label{UNIFDISTRIBTHM}
For $d=0,1,2$, 
\[  \lim_{H\to \infty} \frac{\limHFreq}{H+1} = \frac{1}{3}
\]
\end{theorem}

\begin{proof}
We apply Theorem \ref{AVERAGEFREQTHM}, 
to $L_{d,H+1}$, 
using $m = 3m' + d$ with $3^H\le m' < 3^{H+1}$.
We will need two key identities, each easy to verify, 
\begin{eqnarray*} 
 \hcd{3m'+e,H+1} &=& \left\{
    \begin{array}{ll}
    \hcd{m',H} + 1 & \text{for $e=d$}\\
    \hcd{m',H} & \text{for $e\neq d$}
    \end{array}\right. \\
    \sum_{e=0}^2 [\log_3(3m' + e + 1) - \log_3(3m'+e)]
    &=& \log_3(3m'+3) - \log_3(3m') \\
    &=& \log_3(m'+1) - \log_3(m')
\end{eqnarray*}
With these two ingredients, 
\begin{eqnarray}
    L_{d,H+1} &=& \sum_{m'=3^H}^{3^{H+1}-1} (\hcd{m',H}+1)[\log_3(3m'+d+1)-\log_3(3m'+d)]\nonumber\\
    && + \sum_{m'=3^H}^{3^{H+1}-1}\sum_{e=0, e\neq d}^2 \hcd{m',H}[\log_3(3m'+e+1)-\log_3(3m'+e)]\nonumber\\
     &=& L_{d,H} + \sum_{m'=3^H}^{3^{H+1}-1} \log_3(3m'+d+1)-\log_3(3m'+d) 
  \label{UNIFDISTR2}
\end{eqnarray}
Next, we apply the standard linear approximation
\begin{equation}
\label{UNIFDISTR2}
    \log_3(1+x) = \frac{x}{\ln(3)} + O(x^2)
\end{equation}
valid for all $x>0$, with $x = 1/(3m'+d)$. Replacing $1/(3m'+d)$ by $1/(3m')$ everywhere also makes only a negligible difference, even when summing over all $m'$, as $H$ grows to infinity. Finally, we compare the rewritten sum to an integral which also makes just a negligible difference,
\begin{eqnarray}
\label{UNIFDISTR3}
L_{d,H+1} - L_{d,H} &=& \sum_{m'=3^{H}}^{3^{H+1}-1} \left(\frac{1}{3\ln(3)m'}\right) + O\left(3^{-H}\right)\nonumber\\
    &=& \int_{3^H}^{3^{H+1}}    \frac{1}{3\ln(3) x}\, dx + O\left(3^{-H}\right)\nonumber\\
    &=& \frac{1}{3} + O\left(3^{-H}\right)
\end{eqnarray}
This shows 
\begin{equation}
    \label{UNIFDISTR4}
    \limHFreq = \frac{H}{3} + O(1)
\end{equation}
from which Theorem \ref{UNIFDISTRIBTHM} follows immediately.
\end{proof}

\begin{remark}
\label{UNIFDISTRIBREM}
\begin{enumerate}
\item[a)]
Since the first digit is never zero, and 1 has a higher frequency than 2, there is a certain bias towards 1 and away from 0. 
\item[b)]
Neither Theorem \ref{AVERAGEFREQTHM} nor Theorem \ref{UNIFDISTRIBTHM} say anything about a limit of the \textit{non-aggregate} relative frequency of a digit in the digits of a single power $2^n$ by itself. 
The following section contains the results we know about these questions. 
\item[c)]
For any string $m=1000000\dots 0_3$, the relative frequency of powers $2^n$ with this front end is positive, so such powers must exist for any length of the string of 0s. If we consider only leading strings of fixed length of $2^n$, then the analogue of Erd\"os' conjecture C4 would be false. 
\end{enumerate}
\end{remark}

\section{A Special Case of Baker's Theorem and its Implications}

Baker's Theorem is incredibly general, and some versions give more details about the constants involved. See eg \cite{Baker1975} for the general theorem, and the blog post \cite{Tao2011_1} for the application to our situation. All we need here is this special case. 

\begin{theorem}[Baker 1975 -- very special case]
    \label{Baker75}
    Suppose $a,b$ are algebraic and positive, and $n\ln(a) - m\ln(b) \neq 0$ for all pairs of integers $(m,n)$ except $(0,0)$. 
    Then there exist constants $C, D>0$ such that for all integers $m,n>0$
\[  |n\ln(a) - m\ln(b)| \ge \frac{C}{m^D}
\]
\end{theorem}

Choose $a = 2$, $b = 3$ in Theorem \ref{Baker75}, and write the left-hand side as a single logarithm. Then exponentiate both sides and use the simple fact $e^x > 1 + x$ to get 
\[  \frac{2^n}{3^m} \ge 1 + \frac{C}{m^D}
\]
This gives 
\begin{corollary}[Consequence of Baker's Theorem]
\label{BAKERCOROLLARY}
        There exist constants \(C, D\) such that for all $m,n>0$ with $3^m < 2^n$,
        \[
        2^n - 3^m \geq C\cdot 3^m \cdot m^{-D}
        \]
\end{corollary}
Let us consider this in terms of the ternary digits of $2^n$. If the leading digit is $1$, then this means there can be at most a constant times $\ln(m)$ zeros after the leading digit. But if the leading digit is $2$, followed by a long string of zeros, then we can consider the digits of $2^{n-1}$. The leading digit there would be $1$, followed by a string of zeros of the same or greater length -- hence the number of zeros after the leading digit is $O(\ln(m))$ in all cases.

Corollary \ref{BAKERCOROLLARY} does not contradict part c) of Remark \ref{UNIFDISTRIBREM}, but it narrows down the possibilities for the strings of digits after the leading digit of $2^n$.

\section{Ternary Digits of the Logarithm of 2 to Base 3}

The preceding sections have focused on the properties of the sequence of integers $(2^n)$. We now shift our focus to the properties of a single real number, $\alpha = \log_3(2)$, which already played a key role in our primary investigation. While the distribution of digits in the sequence $(2^n)$ and the distribution of digits in the single number $\alpha$ are distinct problems, they explore a similar theme of apparent randomness in deterministic systems. And of course, the problems are connected: the ternary digits of $\alpha$ contain all the information needed to determine all ternary digits of $2^n$, for every $n$ (see our discussion of Benford's Law in Section \ref{section:BENFORD}). The relationship is particularly straightforward if the exponent $n$ is a power of $3$, say $n = 3^d$. Then 
\[  \log_3(2^n) = n\alpha = 3^d\alpha
\]
In this case, the ternary digits of $\log_3(2^n)$ are simply the same as the leading ternary digits of $\alpha$, shifted $d$ spaces to the left. But even though we know that the sequence $(n\alpha)$ is uniformly distributed modulo 1, this is neither a sufficient nor a necessary condition for the subsequence $(3^d\alpha)$ having this property. 

It still seems natural to investigate the distribution of digits of $\alpha$. 
A central concept for discussing such a digit distribution is that of normality.
\begin{definition}
A real number $x$ is said to be \textit{normal in base b} if, for every positive integer $k$, every possible block of $k$ digits appears in the base-$b$ expansion of $x$ with a limiting frequency of $b^{-k}$ \cite{BaileyCrandall2001}, \cite{Chamberland2003}. A number is \textit{absolutely normal} if it is normal in every integer base $b \ge 2$.
\end{definition}

It is a famous open problem whether $\alpha = \log_3(2)$ is normal to any base. It is widely conjectured that all irrational algebraic numbers and most transcendental constants of interest are absolutely normal, but not a single one has been proven to be normal in even one base \cite{BaileyCrandall2001}.

The modern approach to this problem, pioneered by Bailey and Crandall, connects the normality of certain constants to the behavior of specific chaotic dynamical systems. Their work suggests that constants like $\pi$ and $\ln(2)$ are normal to certain bases, contingent on a powerful conjecture they term 'Hypothesis A'. However, this framework is not known to apply to $\log_3(2)$, as no suitable series representation for it has been discovered. Therefore, its normality remains an open question.

In the spirit of our primary investigation, we conducted a parallel computational analysis of the first $1,000,000$ ternary digits of $\log_3(2)$ to test the conjecture that it is normal to base 3. 


Table \ref{tab:log_digit_freq} shows the frequencies of the individual digits from our computation. 

\begin{table}[h!]
\centering
\begin{tabular}{|c|c|c|}
\hline
\textbf{Digit} & \textbf{Count (out of $10^6$)} & \textbf{Percentage} \\ \hline
\textbf{0} & 334,147 & 33.4147\% \\
\textbf{1} & 332,209 & 33.2209\% \\
\textbf{2} & 333,644 & 33.3644\% \\ \hline
\end{tabular}
\caption{Frequency of the first $1,000,000$ ternary digits of $\log_3(2)$.}
\label{tab:log_digit_freq}
\end{table}

The frequencies are close to the expected value of $33.\overline{3}\%$, although not as close as the values we saw in our investigation of ternary digits of $2^n$.

To test for higher-order uniform distribution, we analyzed the frequencies of strings of length 2. The results, shown in Table \ref{tab:log_string_freq_len2}, are again close to the theoretical value of $1/9 \approx 11.111\%$.

\begin{table}[h!]
\centering
\begin{tabular}{|c|c||c|c|}
\hline
\textbf{String} & \textbf{Frequency} & \textbf{String} & \textbf{Frequency} \\ \hline
'00' & 11.1758\% & '12' & 11.0796\% \\
'01' & 11.1590\% & '20' & 11.0802\% \\
'02' & 11.1472\% & '21' & 11.0794\% \\
'10' & 11.0914\% & '22' & 11.1712\% \\
'11' & 11.0162\% & & \\
\hline
\end{tabular}
\caption{Frequency of 2-digit strings from the first $1,000,000$ ternary digits of $\log_3(2)$.}
\label{tab:log_string_freq_len2}
\end{table}

Finally, an analysis of 3-digit strings also showed strong convergence to the expected frequency of $1/27 \approx 3.7037\%$, further supporting the conjecture that $\log_3(2)$ is normal to base 3. 
It is intriguing that both the sequence of digit distributions for $(2^n)$ and the digit distribution for the single number $\log_3(2)$ show such strong computational evidence of uniformity, even if a precise theoretical bridge remains to be built.

\section*{Conclusion}

Extensive computations show that the ternary digits of $2^n$ exhibit striking uniformity: aggregate frequencies for digits and short blocks converge rapidly to the expected values $1/3$ and $3^{-k}$, with deviations not much bigger than under naive independence. Parallel results for the ternary expansion of $\log_3 2$ display similar behavior. These findings strongly support the conjectured uniform distribution. 

Mathematical proofs for all of these remain elusive, we can only prove results for the aggregate distribution of digits in the 'front end' of powers of 2. Baker's Theorem gives an upper bound for runs of 0s after the leading digit which is the only pertinent result we know that is valid for individual powers of 2. The frequency of 0s is particularly interesting, though, because of a connection to Selfridge's conjecture, which is still open. 
This conjecture is briefly stated as follows. Define for every integer $A$ the \textit{integer complexity} $||A||$ as the minimal number of 1s which allows to express $A$ using addition and multiplication. Then Selfridge conjectured $||2^n|| = 2n$ (obviously, $2^n = (1+1)(1+1)\dots(1+1)$, so $||2^n||\le 2n$). 

For details, we refer the reader to Altman/Zelinsky \cite{Zelinsky}. Let us just conclude with the remark from that paper that a counterexample to Selfridge's conjecture would need to involve a power of 2 with 'many zeros' in its ternary digits. 

\vskip 20pt
\noindent {\bf Acknowledgements.} 
The authors are grateful to Marc Chamberland, Christy Hazel (Grinnell), and Jonathan DH Smith (Iowa State University) for valuable comments and suggestions, as well as for invitations to present in their respective seminars. 
The Mathematics Department of Grinnell College provided support in the form of a summer internship, and the Mathematics Department at Iowa State provided travel support. 


\end{document}